\documentclass[11pt,leqno]{amsart}

\usepackage{amsmath}
\usepackage{amssymb}
\usepackage{a4wide}
\usepackage{graphics}
\usepackage{epsfig}

\newtheorem{theorem}{Theorem}[section]
\newtheorem{lemma}[theorem]{Lemma}
\newtheorem{proposition}[theorem]{Proposition}

\newtheorem{remark}[theorem]{Remark}

\newcommand{\Subsection}[1]{\subsection{ #1} ${}^{}$}

\newcounter{hypo}

\def\C{{\mathbb C}}
\def\N{{\mathbb N}} 
\def\R{{\mathbb R}}

\def\S{{\mathbb S}}

\def\CA{\mathcal {A}}

\def\CM{\mathcal {M}}

\def\CO{\mathcal {O}}

\def\re{\mathop{\rm Re}\nolimits}
 \def\im{\mathop{\rm Im}\nolimits}

\def\supp{\mathop {\rm supp}\nolimits}

\def\Hess{\mathop{\rm Hess}\nolimits}
\def\<{\langle}
\def\>{\rangle}

\newcommand{\fract}[2]{\genfrac{}{}{0pt}{}{\scriptstyle #1}{\scriptstyle #2}}

\makeatletter
 \@addtoreset{equation}{section}
 \makeatother

\author{Jean-Fran\c{c}ois Bony}
\address{\newline Institut de Math\'ematiques de Bordeaux   \newline UMR 5251 du CNRS   \newline Universit\'e de Bordeaux I   \newline 351 cours de la Lib\'eration   \newline 33 405 Talence cedex    \newline France}
\email{bony@math.u-bordeaux1.fr}
\author{Dietrich H\"{a}fner}
\email{hafner@math.u-bordeaux1.fr}

\title[Decay and non-decay in the De~Sitter--Schwarzschild metric]{Decay and non-decay of the local energy for the wave equation in the De~Sitter--Schwarzschild metric}



\begin{document}

\begin{abstract}
We describe an expansion of the solution of the wave equation in the De Sitter--Schwarzschild metric in terms of resonances. The main term in the expansion is due to a zero resonance. The error term decays polynomially if we permit a logarithmic derivative loss in the angular directions and exponentially if we permit an $\varepsilon$ derivative loss in the angular directions. 
\end{abstract}

\maketitle


\section{Introduction}

There has been important progress in the question of local energy
decay for the solution of the wave equation in black hole type
space-times other the last years. The best results are now known in
the Schwarzschild space-time. We refer to the papers of Blue--Soffer
\cite{BlSo06_01}, Blue--Sterbenz \cite{BlSt06_01}, and Dafermos--Rodnianski \cite{DaRo05_01} and references therein for an
overview. See also the paper of Finster--Kamran--Smoller--Yau for
the Kerr space-time \cite{FKSY06_01}. Results on the decay of local energy
are believed to be a prerequisite for a possible proof of
the global nonlinear stability of these space-times. Today global
nonlinear stability is only known for the Minkowski space-time (see
\cite{ChK93_01}).

From our point of view one of the most efficient approaches to the
question of local energy decay is the theory of resonances. Resonances
correspond to the frequencies and rates of dumping of signals emitted
by the black hole in the presence of perturbations (see
\cite[Chapter 4.35]{ch92_01}). On the one hand these resonances are today
an important hope of effectively detecting the presence of a black hole as
we are theoretically able to measure the corresponding gravitational
waves. On the other hand, the distance of the resonances to the real
axis reflects the stability of the system under the perturbation:
larger distances correspond to more stability. In particular the
knowledge of the localization of resonances permits to answer the
question if there is decay of local energy and at which rate. The aim
of the present paper is to show how this method applies to the
simplest model of a black hole: the De Sitter--Schwarzschild black
hole.

In the euclidean space, such results are already known, especially in the non trapping geometry. The first result is due to Lax and Phillips (see their book \cite[Theorem III.5.4]{LaPh89_01}). They have proved that the cut-off propagator associated to the wave equation outside an obstacle in odd dimension $\geq 3$ (more precisely the Lax--Phillips semi-group $Z (t)$) has an expansion in terms of resonances if $Z(T)$ is compact for one $T$. In particular, the local energy decays exponentially uniformly. From Melrose--Sj\"{o}strand \cite{MeSj78_01}, this assumption is true for non trapping obstacles. Va{\u\i}nberg \cite{Va89_01} has obtained such results for general, non trapping, differential operators using different techniques. In the trapping case, we know, by the work of Ralston \cite{Ra69_01}, that it is not possible to obtain uniform decay estimates without loss of derivatives. In the exterior domain of two strictly convex obstacles, the local energy decays exponentially with a loss of derivatives, by the work of Ikawa \cite{Ik82_01}. This situation is close to the one treated in this paper. We also mention the works Tang--Zworski \cite{TaZw00_01} and Burq--Zworski \cite{BuZw01_01} concerning the resonances close to the real line.

Thanks to the work of S\'a Barreto and Zworski (\cite{SaZw97_01}) we
have a very good knowledge of the localization of resonances for the
wave equation in the De Sitter--Schwarzschild metric. Using their
results we can describe an expansion of the solution of the wave
equation in the De Sitter--Schwarzschild metric in terms of
resonances. The main term in the expansion is due to a zero
resonance. The error term decays polynomially if we permit a
logarithmic derivative loss in the angular directions and
exponentially if we permit an $\varepsilon$ derivative loss in the
angular directions. For initial data in the complement of a
one-dimensional space the local energy is integrable if we permit a
$(\ln \langle -\Delta_{\omega} \rangle)^{\alpha}$ derivative loss with
$\alpha>1$. This estimate is almost optimal in the sense that it
becomes false for $\alpha<\frac{1}{2}$. 

The method presented in this paper does not directly apply to the
Schwarzschild case. This is not linked to the difficulty of the photon
sphere which we treat in this paper, but to the possible
accumulation of resonances at the origin in the Schwarzschild case.

The exterior of the De Sitter--Schwarzschild black hole is given by
\begin{gather}
( \CM , g),  \quad \CM = \R_{t} \times X \text{ with } X = ] r_{-} , r_{+}  [_{r} \times \S^{2}_{\omega} \label{a1}  \\
g = \alpha^{2} d t^{2} - \alpha^{-2} d r^{2} - r^{2} \vert d \omega \vert^{2}, \quad \alpha = \Big( 1 - \frac{2 M}{r} - \frac{1}{3} \Lambda r^{2} \Big)^{1/2},  \label{a2}
\end{gather}
where $M>0$ is the mass of the black holes and $0 < 9 M^{2} \Lambda < 1$ is the cosmological constant. $r_{-}$ and $r_{+}$ are the two positive roots of $\alpha =0$. We also denoted by $\vert d \omega \vert^{2}$ the standard metric on $\S^{2}$.

The corresponding d'Alembertien is
\begin{equation} \label{a3}
\Box_{g} = \alpha^{-2} \big( D_{t}^{2} - \alpha^{2}r^{-2} D_{r} (r^{2} \alpha^{2}) D_{r} + \alpha^{2} r^{-2} \Delta_{\omega} \big),
\end{equation}
where $D_{\bullet} = \frac{1}{i} \partial_{\bullet}$ and $- \Delta_{\omega}$ is the positive Laplacian on $\S^{2}$.
We also denote
\begin{equation*}
\widehat{P} = \alpha^{2} r^{-2} D_{r} (r^{2} \alpha^{2} ) D_{r} - \alpha^{2} r^{-2} \Delta_{\omega} ,
\end{equation*}
the operator on $X$ which governs the situation on $L^{2} ( X, r^{2} \alpha^{-2} d r \, d \omega )$. We define
\begin{equation*}
P = r \widehat{P} r^{-1},
\end{equation*}
on $L^{2} ( X, \alpha^{-2} d r \, d \omega )$, and, in the coordinates $(r, \omega )$, we have
\begin{equation*}
P = \alpha^{2} D_{r} ( \alpha^{2} D_{r}) - \alpha^{2}r^{-2} \Delta_{\omega} + r^{-1} \alpha^{2} (\partial_{r} \alpha^{2}) .
\end{equation*}

We introduce the Regge--Wheeler coordinate given by
\begin{equation}  \label{a7}
x' (r) = \alpha^{-2}
\end{equation}
In the coordinates $(x , \omega)$, the operator $P$ is given by
\begin{equation}  \label{a8}
P = D_{x}^{2} - \alpha^{2} r^{-2} \Delta_{\omega} + \alpha^{2} r^{-1} ( \partial_{r} \alpha^{2} )
\end{equation}
on $L^{2} ( X, d x \, d \omega )$. Let $V = \alpha^{2} r^{-2}$ and $W = \alpha^{2} r^{-1} ( \partial_{r} \alpha^{2} )$ be the potentials appearing in the previous operator. As stated in Proposition 2.1 of \cite{SaZw97_01}, the work of Mazzeo--Melrose \cite{MaMe87_01} implies that for $\chi\in C_0^{\infty}(\R)$  
\begin{equation*}
R_{\chi} ( \lambda ) = \chi(P-\lambda^{2} )^{-1}\chi ,
\end{equation*}
has a meromorphic extension from the upper half plane to $\C$. The poles $\lambda$ of this meromorphic extension are called resonances. We recall the principal result of \cite{SaZw97_01}:

\begin{theorem}[{\bf S\'a Barreto--Zworski}]\sl \label{a52}
There exists $K>0$ and $\theta>0$ such that for any $C>0$ there exists an injective map, $\tilde{b}$, from the set of pseudo-poles
\begin{eqnarray*}
\frac{(1-9\Lambda M^2)^{\frac{1}{2}}}{3^{\frac{3}{2}}M}\left(\pm \N\pm \frac{1}{2}-i\frac{1}{2}\left(\N_0+\frac{1}{2}\right)\right) ,
\end{eqnarray*}
into the set of poles of the meromorphic continuation of $(P-\lambda^2)^{-1} : L^{2}_{\rm comp} \to L^{2}_{\rm loc}$ such that all the poles in 
\begin{eqnarray*}
\Omega_C=\{\lambda : \, \im \lambda>-C,\, |\lambda|>K,\, \im \lambda>-\theta |\re\lambda| \} ,
\end{eqnarray*}
are in the image of $\tilde{b}$ and for $\tilde{b}(\mu)\in \Omega_C$,
\begin{eqnarray*}
\tilde{b}(\mu)-\mu\rightarrow 0\quad\mbox{as}\quad |\mu|\rightarrow \infty.
\end{eqnarray*}
If $\mu= \mu_{\ell , j}^{\pm} = 3^{-\frac{3}{2}}M^{-1}(1-9\Lambda M^2)^{\frac{1}{2}} \big( (\pm \ell \pm \frac{1}{2}) - i \frac{1}{2} (j + \frac{1}{2})\big)$, $\ell \in \N$, $j \in \N_{0}$, then the corresponding pole, $\tilde{b}(\mu)$, has multiplicity $2\ell +1$.
\end{theorem}
\begin{figure}[!h]
\begin{picture}(0,0)%
\includegraphics{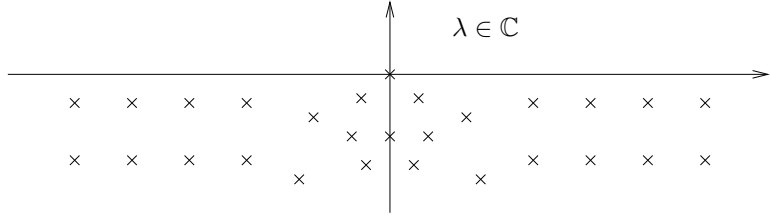}%
\end{picture}%
\setlength{\unitlength}{1579sp}%
\begingroup\makeatletter\ifx\SetFigFont\undefined%
\gdef\SetFigFont#1#2#3#4#5{%
  \reset@font\fontsize{#1}{#2pt}%
  \fontfamily{#3}\fontseries{#4}\fontshape{#5}%
  \selectfont}%
\fi\endgroup%
\begin{picture}(12044,3419)(-21,-3158)
\put(6976,-361){\makebox(0,0)[lb]{\smash{{\SetFigFont{10}{12.0}{\rmdefault}{\mddefault}{\updefault}$\lambda \in \C$}}}}
\end{picture}%
\caption{The resonances of $P$ near the real axis.}
\label{reso}
\end{figure}

The natural energy space ${\mathcal E}$ for the wave equation is given by the completion 
of $C_0^{\infty}(\R\times \S^2)\times C_0^{\infty}(\R\times \S^2)$ in the norm 
\begin{equation}
\label{ES}
\Vert(u_0,u_1)\Vert_{\mathcal E}^2=\Vert u_1\Vert^2+\langle P u_0,u_0 \rangle.
\end{equation}
It turns out that this is not a space of distributions. The problem is very similar to the problem for the wave equation in dimension $1$. We therefore introduce another energy space ${\mathcal E}_{a,b}^{\rm mod} \, (-\infty<a<b<\infty)$ defined as the completion of $C_0^{\infty}(\R\times \S^2)\times C_0^{\infty}(\R\times \S^2)$ in the norm
\begin{eqnarray*}
\Vert(u_0,u_1)\Vert^2_{{\mathcal E}^{\rm mod} }=\Vert u_1\Vert^2+\langle Pu_0,u_0\rangle +\int_a^b\int_{\S^2} \vert u_{0} (s,\omega) \vert^2 d s \, d\omega.
\end{eqnarray*}
Note that for any $-\infty<a<b<\infty$ the norms ${\mathcal E}^{\rm mod}_{a,b}$ and ${\mathcal E}^{\rm mod}_{0,1}$ are equivalent. We will therefore only work with the space ${\mathcal E}^{\rm mod}_{0,1}$ in the future and note it from now on ${\mathcal E}^{\rm mod}$.
Let us write the wave equation as a first order equation in the following way:
\begin{eqnarray*}
\left\{ \begin{aligned} &i\partial_t v = Lv  \\ 
&v(0) = v_0 \end{aligned} \right. \quad \text{with} \quad
L=\left(\begin{array}{cc} 0 & i\\ -iP & 0 \end{array} \right).
\end{eqnarray*}
Let ${\mathcal H}^k$ be the scale of Sobolev spaces associated to $P$. We note ${\mathcal H}^2_c$ the completion of ${\mathcal H}^2$ in the norm $\Vert u\Vert_2^2:=\langle Pu,u\rangle+\Vert Pu\Vert^2.$ Then $(L,D(L)={\mathcal H}^2_c\oplus{\mathcal H}^1)$ is selfadjoint on ${\mathcal E}$. We note ${\mathcal E}^k$ the scale of Sobolev spaces associated to $L$. Note that because of 
\begin{eqnarray}
\label{RL}
(L-\lambda)^{-1}=(P-\lambda^2)^{-1}\left(\begin{array}{cc} \lambda & i \\ -iP & \lambda \end{array} \right)
\end{eqnarray}
we can define a meromorphic extension of the cut-off resolvent of $L$ by using the meromorphic extension of the cut-off resolvent of $P$ and the resonances of $L$ coincide with the resonances of $P$.

Recall that $(-\Delta_{\omega}, H^2(\S^2))$ is a selfadjoint operator with compact resolvent. Its eigenvalues are the $\ell(\ell +1),\, \ell \ge 0$ with multiplicity $2 \ell +1$. We denote
\begin{equation} \label{a70}
P_\ell =r^{-1}D_x r^2 D_x r^{-1}+\alpha^2r^{-2}\ell(\ell+1) 
\end{equation}
the operator restricted to ${\mathcal H}_\ell =L^2(\R)\times Y_\ell$ where $Y_\ell$ is the eigenspace to the eigenvalue $\ell(\ell +1)$. In the following, $P_\ell$ will be identify with the operator on $L^{2} (\R)$ given by \eqref{a70}. The spaces ${\mathcal E}_{\ell},\, {\mathcal E}^{\rm mod}_{\ell},\, {\mathcal E}^k_{\ell}$ are defined in an analogous way to the spaces ${\mathcal E},\, {\mathcal E}^{\rm mod},\, {\mathcal E}^k$. Let $\Pi_{\ell}$ be the orthogonal projector on~${\mathcal E}_{\ell}^{\rm mod}$. For $\ell \geq 1$, the space ${\mathcal E}^{\rm mod}_{\ell}$ and ${\mathcal E}_{\ell}$ are the same and the norms are uniformly equivalent with respect to $\ell$.

Using Proposition II.2 of Bachelot and Motet-Bachelot \cite{BaMo93_01}, the group $e^{-i t L}$ preserves the space ${\mathcal E}^{\rm mod}$ and there exists $C , k >0$ such that
\begin{equation*}
\Vert e^{-i t L} u \Vert_{{\mathcal E}^{\rm mod}} \leq C e^{k \vert t \vert} \Vert u \Vert_{{\mathcal E}^{\rm mod}}.
\end{equation*}
From the previous discussion, the same estimate holds with $k =0$ uniformly in $\ell \geq 1$. In particular, $(L - z)^{-1}$ is bounded on ${\mathcal E}^{\rm mod}$ for $\im z > k$, and we note ${\mathcal E}^{\rm mod , -j} = (L - z)^{j} {\mathcal E}^{\rm mod} \subset {\mathcal D}' (\R \times \S^{2})$ for $j \in \N_{0}$.

We first need a result on $P$:

\begin{proposition}\sl  \label{a50}
For $\ell \geq 1$, the operator $P_{\ell}$ has no resonance and no eigenvalue on the real axis.

For $\ell =0$, $P_{0}$ has no eigenvalue in $\R$ and no resonance in $\R \setminus \{ 0 \}$. But, $0$ is a simple resonance of $P_{0}$, and, for $z$ closed to $0$, we have
\begin{equation} \label{a64}
(P_{0} - z^{2})^{-1} = \frac{i \gamma}{z} r \< r \vert \ \cdot \ \> + H (z),
\end{equation}
where $\gamma \in ]0 , + \infty [$ and $H (z)$ is a holomorphic (bounded) operator near $0$. Equation \eqref{a64} is an equality between operators from $L^{2}_{{\rm comp}}$ to $L^{2}_{{\rm loc}}$.
\end{proposition}

The proof of Proposition \ref{a50} is given in Section \ref{a51}. For $\chi\in C_0^{\infty}(\R)$ we put in the following:
\begin{equation*}
\widehat{R}_{\chi}(\lambda)=\chi(L-\lambda)^{-1}\chi.
\end{equation*}
For a resonance $\lambda_j$ we define $m(\lambda_j)$ by the Laurent
expansion of the cut-off resolvent near $\lambda_j$:
\begin{eqnarray*}
\widehat{R}_{\chi}(\lambda)=\sum_{k=-(m(\lambda_j)+1)}^{\infty}A_k(\lambda-\lambda_j)^k.
\end{eqnarray*}
We also define $\pi^{\chi}_{j,k}$ by 
\begin{eqnarray}
\label{pr}
\pi^{\chi}_{j,k}=\frac{-1}{2 \pi i}\oint\frac{(-i)^k}{k!}\widehat{R}_{\chi}(\lambda)(\lambda-\lambda_j)^k d\lambda.
\end{eqnarray}
The main result of this paper is the following:

\begin{theorem}\sl  \label{mainth}
Let $\chi\in C_0^{\infty}(\R)$. 

$(i)$ Let
$0<\mu\notin
\frac{(1-9\Lambda M^2)^{1/2}}{3^{1/2} M}\frac{1}{2}\left(\N_0+\frac{1}{2}\right)$
such that there is no resonance with $\im z=-\mu$. Then there exists
$M>0$ with the following property. Let $u\in
{\mathcal E}^{\rm mod}$ such that $\langle -\Delta_{\omega}\rangle^M u\in
{\mathcal E}^{\rm mod}$. Then we have:
\begin{equation}
\label{RE1}
\chi e^{-itL}\chi u = \sum_{\fract{\lambda_j\in {\rm Res} \, P}{\im
    \lambda_j>-\mu}}\sum_{k=0}^{m(\lambda_j)}e^{-i\lambda_j
  t}t^k\pi_{j,k}^{\chi}u+E_1(t) ,
\end{equation}
with
\begin{equation} \label{a68}
\Vert E_1(t)\Vert_{{\mathcal E}^{\rm mod}} \lesssim e^{-\mu t}\Vert\langle
-\Delta_{\omega}\rangle^Mu\Vert_{{\mathcal E}^{\rm mod}} ,
\end{equation}
and the sum is absolutely convergent in the sense that 
\begin{eqnarray}
\label{sp}
\sum_{\fract{\lambda_j\in {\rm Res} \, P}{\im
    \lambda_j>-\mu}}\sum_{k=1}^{m(\lambda_j)}\Vert\pi_{j,k}^{\chi}\langle -\Delta_{\omega} \rangle^{-M}\Vert_{{\mathcal L}({\mathcal E}^{\rm mod} )}\lesssim 1.
\end{eqnarray}

$(ii)$ There exists $\varepsilon>0$ with the following property. Let $g\in C([0, + \infty[),\, 
\lim_{|x|\rightarrow \infty}g(x)=0$, positive, strictly decreasing with $x^{-1}\le g(x)$ for $x$ large. Let $u=(u_1,u_2)\in {\mathcal E}^{\rm mod}$ be such that $\big( g( -\Delta_{\omega} ) \big)^{-1} u \in {\mathcal E}^{\rm mod}$. Then we have      
\begin{equation}  \label{meq1}
\chi e^{-itL}\chi u = \gamma \left(\begin{array}{c} r\chi\langle r, \chi u_2 \rangle \\ 0 \end{array} \right)+E_2(t)u ,
\end{equation}
with
\begin{equation}  \label{miq1}
\Vert E_2(t)u\Vert_{{\mathcal E}^{\rm mod}} \lesssim g(e^{\varepsilon  t}) \big\Vert \big( g( -\Delta_{\omega} ) \big)^{-1} u \big\Vert_{{\mathcal E}^{\rm mod}} .
\end{equation}
\end{theorem}

\begin{remark}\rm \label{a69}
a) By the results of S\'a Barreto and Zworski we know that there exists
$\mu>0$ such that $0$ is the only resonance in $\im z>-\mu$. Choosing
this $\mu$ in $(i)$ the sum on the right hand side contains a single element which is  
\begin{eqnarray*}
\gamma \left(\begin{array}{c} r\chi\langle r, \chi u_2 \rangle \\
0 \end{array} \right).
\end{eqnarray*}

b) Again by the paper of S\'a Barreto and Zworski we know that $\lambda_{j} = \widetilde{b} ( \mu_{\ell , \widetilde{\jmath}}^{\varepsilon} )$ for all the $\lambda_j$'s outside a compact set (see Theorem \ref{a52}). For such $\lambda_{j}$, we have $m_j(\lambda_j)=0$ and $\pi^{\chi}_{j ,k} = \Pi_{\ell} \pi^{\chi}_{j ,k} = \pi^{\chi}_{j ,k} \Pi_{\ell}$ is an operator of rank $2 \ell +1$.

c) Let ${\mathcal E}^{{\rm mod},\perp}=\{u\in{\mathcal E}^{\rm mod} ; \ \langle r,\chi u_2\rangle =0 \}$.
By part $(ii)$ of the theorem, for $u \in {\mathcal E}^{{\rm mod}, \perp}$, the local energy is integrable if $( \ln \< - \Delta_{\omega} \> )^{\alpha} u \in {\mathcal E}^{\rm mod}$, for some $\alpha >1$, and decays exponentially if $\< - \Delta_{\omega} \>^{\varepsilon} u \in {\mathcal E}^{\rm mod}$ for some $\varepsilon >0$.

d) In fact, we can replace $\< - \Delta_{\omega} \>^{M}$ by $\< P \>^{2M}$ in the first part of the theorem. And, by an interpolation argument, we can obtain the following estimate: for all $\varepsilon >0$, there exists $\delta >0$ such that
\begin{equation}  \label{a66}
\chi e^{-i t L}\chi u = \gamma \left(\begin{array}{c} r\chi\langle r, \chi u_2 \rangle \\ 0 \end{array} \right)+E_3(t)u,
\end{equation}
with
\begin{equation}
\Vert E_3 (t)u\Vert_{{\mathcal E}^{\rm mod}} \lesssim e^{-\delta  t} \big\Vert \< P \>^{\varepsilon} u \big\Vert_{{\mathcal E}^{\rm mod}}.
\end{equation}
\end{remark}

\begin{remark}\rm
In the Schwarzschild case the potential $V (x)$ is only polynomially decreasing at infinity and we
cannot apply the result of Mazzeo--Melrose. Therefore we cannot exclude 
a possible accumulation of resonances at $0$. This
difficulty has nothing to do with the presence of the photon sphere
which is treated by the method presented in this paper.
\end{remark}

\begin{remark}\rm
Let $u\in {\mathcal E}^{{\rm mod} ,\perp}$ such that $(\ln\langle
-\Delta_{\omega} \rangle)^{\alpha}u\in {\mathcal E}^{\rm mod}$ for some
$\alpha>1$. 
Then we have from part $(ii)$ of the theorem, for $\lambda \in \R$,
\begin{equation}  \label{LEE}
\Big\Vert \int_0^{\infty} \chi e^{-i t (L- \lambda )}\chi u \, d t \Big\Vert_{{\mathcal E}^{\rm mod}}\lesssim\Vert(\ln \langle -\Delta_{\omega}
\rangle)^{\alpha} u \Vert_{{\mathcal E}^{\rm mod}}.
\end{equation}
This estimate is almost optimal since it becomes false for
$\alpha<\frac{1}{2}$. Indeed we have ($\lambda \in \R$):
\begin{eqnarray*}
\widehat{R}_{\chi}( \lambda )u= i \int_0^{\infty}\chi e^{-it(L-\lambda )}\chi u \, d t.
\end{eqnarray*}
Thus from (\ref{LEE}) we obtain the resolvent estimate 
\begin{eqnarray*}
\Vert\widehat{R}_{\chi}( \lambda )(\ln\langle
-\Delta_{\omega}\rangle)^{-\alpha}\Vert_{{\mathcal L}({\mathcal
    E}^{\rm mod})}\lesssim 1.
\end{eqnarray*}
It is easy to see that this entails the resolvent estimate 
\begin{eqnarray*}
\Vert R_{\chi}(\lambda )(\ln\langle-\Delta_{\omega}\rangle)^{-\alpha}\Vert \lesssim \frac{1}{\vert \lambda \vert}.
\end{eqnarray*}
We introduce the semi-classical parameter $h^2=( \ell ( \ell +1))^{-1}$ and
$\widetilde{P}=-h^2D_x^2+V(x)+h^{2} W(x)$ as in Section \ref{a65}. Then, for $R>0$, the above estimate gives the
semi-classical estimate:
\begin{eqnarray*}
\Vert\chi( \widetilde{P} - z)^{-1}\chi\Vert\lesssim\frac{\vert \ln h \vert^{\alpha}}{h},
\end{eqnarray*}
for $1/R \leq z \leq R$ (see \eqref{a33} and \eqref{a34}). Such estimate is known to be false for $\alpha<\frac{1}{2}$ and $z = z_{0}$, the maximum value of the
potential $V (x)$ (see \cite[Proposition 2.2]{AlBoRa07_01}).
\end{remark}

The proof of the theorem is based on resolvent estimates. 
Using (\ref{RL}) we see that it is sufficient to prove resolvent
estimates for $\chi(P_{\ell} -\lambda^2)^{-1}\chi$. This is the purpose of the next section.

\textbf{Acknowledgments:} We would like to thank A. Bachelot for fruitful discussions during the preparation of this article. This work was partially supported by the ANR project JC0546063 ``Equations hyperboliques dans les espaces temps de la relativit\'e g\'en\'erale : diffusion et r\'esonances''.

\section{Estimate for the cut-off resolvent.}

In this section, we obtain estimates for the cut-off resolvent of $P_{\ell}$, the operator $P$ restricted to the spherical harmonic $\ell$. We will use the description of the resonances given in S{\'a}~Barreto--Zworski \cite{SaZw97_01}. Recall that
\begin{equation}
R_{\chi} (\lambda ) = \chi (P - \lambda^{2} )^{-1} \chi ,
\end{equation} 
has a meromorphic extension from the upper half plane to $\C$. The resonances of $P$ are defined as the poles of this extension. We treat only the case $\re \lambda > - 1$ since we can obtain the same type of estimates for $\re \lambda < 1$ using $\big( R_{\chi} (- \overline{\lambda}) \big)^{*} = R_{\chi} (\lambda )$.

\begin{theorem}\sl \label{a44}
Let $C_{0} > 0$ be fixed. The operators $\chi (P_{\ell} - \lambda^{2} )^{-1} \chi$ satisfy the following estimates uniformly in $\ell$.

i) For all $R>0$, the number of resonances of $P$ is bounded in $B (0,R)$. Moreover, there exists $C>0$ such that
\begin{equation}  \label{a39}
\Vert \chi (P_{\ell} - \lambda^{2} )^{-1} \chi \Vert \leq \Vert \chi (P - \lambda^{2} )^{-1} \chi \Vert \leq C \prod_{\fract{\lambda_{j} \in {\rm Res} \, P }{\vert \lambda_{j} \vert < 2R}} \frac{1}{\vert \lambda - \lambda_{j} \vert}
\end{equation}
for all $\lambda \in B (0,R)$.

ii) For $R$ large enough, $P_{\ell}$ has no resonance in $[R, \ell /R] + i [- C_{0} ,0]$. Moreover, there exists $C>0$ such that
\begin{equation}  \label{a40}
\Vert \chi (P_{\ell} - \lambda^{2} )^{-1} \chi \Vert \leq \frac{C}{\< \lambda \>^{2}} ,
\end{equation}
for $\lambda \in [R, \ell /R] + i [- C_{0} , C_{0} ]$.

iii) Let $R$ be fixed. For $\ell$ large enough, the resonances of $P_{\ell}$ in $[\ell /R , R \ell] + i[-C_{0} ,0]$ are the $\widetilde{b} ( \mu_{\ell ,j}^{+})$ given in Theorem \ref{a52} (in particular their number is bounded uniformly in $\ell$). Moreover, there exists $C>0$ such that
\begin{equation} \label{a41}
\Vert \chi (P_{\ell} - \lambda^{2} )^{-1} \chi \Vert \leq C \< \lambda \>^{C} \prod_{\fract{\lambda_{j} \in {\rm Res} \, P_{\ell}}{\vert \lambda - \lambda_{j} \vert < 1}} \frac{1}{\vert \lambda - \lambda_{j} \vert} ,
\end{equation}
for $\lambda \in [ \ell /R, R \ell ] + i [- C_{0} ,C_{0} ]$.

Furthermore, $P_{\ell}$ has no resonance in $[ \ell /R, R \ell ] + i [- \varepsilon ,0]$, for some $\varepsilon >0$, and we have
\begin{equation}  \label{a42}
\Vert \chi (P_{\ell} - \lambda^{2})^{-1} \chi \Vert \leq C \frac{\ln \< \lambda \>}{\< \lambda \>} e^{C \vert \im \lambda \vert \ln \< \lambda \>} ,
\end{equation}
for $\lambda \in [ \ell /R, R \ell ] + i [- \varepsilon ,0]$.

iv) Let $C_{1} >0$ be fixed. For $R$ large enough, $P_{\ell}$ has no resonance in $\{ \lambda \in \C; \ \re \lambda > R \ell , \text{ and } 0 \geq \im \lambda \geq - C_{0} - C_{1} \ln \< \lambda \> \}$. Moreover, there exists $C>0$ such that
\begin{equation}  \label{a43}
\Vert \chi (P_{\ell} - \lambda^{2})^{-1} \chi \Vert \leq \frac{C}{\< \lambda \>} e^{C \vert \im \lambda \vert} ,
\end{equation}
for $\re \lambda >R \ell$ and $C_{0} \geq \im \lambda \geq - C_{0} - C_{1} \ln \< \lambda \>$.
\end{theorem}

The results concerning the localization of the resonances in this theorem are proved in \cite{BaMo93_01} and \cite{SaZw97_01}, the following figure summaries the different estimates of the resolvent.
\begin{figure}[!h]
\begin{picture}(0,0)%
\includegraphics{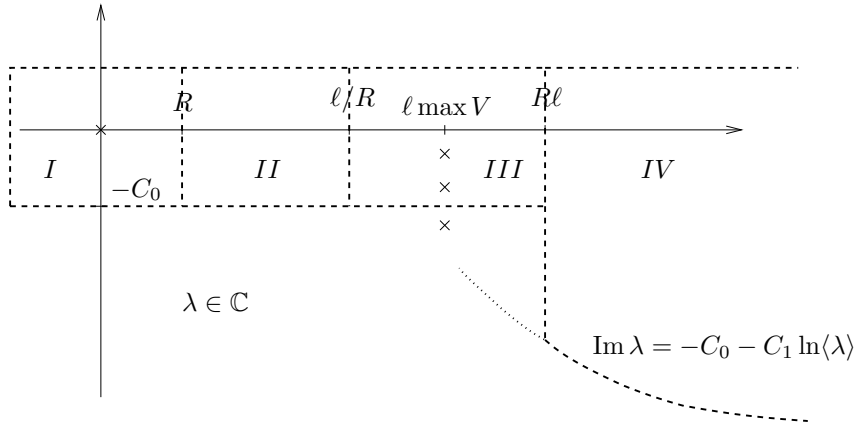}%
\end{picture}%
\setlength{\unitlength}{1579sp}%
\begingroup\makeatletter\ifx\SetFigFont\undefined%
\gdef\SetFigFont#1#2#3#4#5{%
  \reset@font\fontsize{#1}{#2pt}%
  \fontfamily{#3}\fontseries{#4}\fontshape{#5}%
  \selectfont}%
\fi\endgroup%
\begin{picture}(12591,6655)(418,-6769)
\put(9601,-5686){\makebox(0,0)[lb]{\smash{{\SetFigFont{10}{12.0}{\rmdefault}{\mddefault}{\updefault}$\im \lambda = -C_{0} - C_{1} \ln \< \lambda \>$}}}}
\put(976,-2911){\makebox(0,0)[lb]{\smash{{\SetFigFont{10}{12.0}{\rmdefault}{\mddefault}{\updefault}$I$}}}}
\put(4276,-2911){\makebox(0,0)[lb]{\smash{{\SetFigFont{10}{12.0}{\rmdefault}{\mddefault}{\updefault}$II$}}}}
\put(7876,-2911){\makebox(0,0)[lb]{\smash{{\SetFigFont{10}{12.0}{\rmdefault}{\mddefault}{\updefault}$III$}}}}
\put(10351,-2911){\makebox(0,0)[lb]{\smash{{\SetFigFont{10}{12.0}{\rmdefault}{\mddefault}{\updefault}$IV$}}}}
\put(3151,-5011){\makebox(0,0)[lb]{\smash{{\SetFigFont{10}{12.0}{\rmdefault}{\mddefault}{\updefault}$\lambda \in \C$}}}}
\put(2026,-3211){\makebox(0,0)[lb]{\smash{{\SetFigFont{10}{12.0}{\rmdefault}{\mddefault}{\updefault}$-C_{0}$}}}}
\put(3001,-1861){\makebox(0,0)[lb]{\smash{{\SetFigFont{10}{12.0}{\rmdefault}{\mddefault}{\updefault}$R$}}}}
\put(5476,-1786){\makebox(0,0)[lb]{\smash{{\SetFigFont{10}{12.0}{\rmdefault}{\mddefault}{\updefault}$\ell / R$}}}}
\put(8626,-1786){\makebox(0,0)[lb]{\smash{{\SetFigFont{10}{12.0}{\rmdefault}{\mddefault}{\updefault}$R \ell$}}}}
\put(6601,-1936){\makebox(0,0)[lb]{\smash{{\SetFigFont{10}{12.0}{\rmdefault}{\mddefault}{\updefault}$\ell \max V$}}}}
\end{picture}%
\caption{The different zones in Theorem \ref{a44}.}
\label{zones}
\end{figure}

In zone {\it I} which is compact, the result of Mazzeo--Melrose \cite{MaMe87_01} gives a bound uniform with respect to $\ell$ (outside of the possible resonances). In particular, part {\it i)} of Theorem \ref{a44} is a direct consequence of this work.

In zone {\it II}, the result of Zworski \cite{Zw99_01} gives us a good (uniform with respect to $\ell$) estimate of the resolvent. Here, we use the exponential decay of the potential at $+ \infty$ and $- \infty$. By comparison, the equivalent potential for the Schwarzschild metric does not decay exponentially, and our present work cannot be extended to this setting. Please note that this problem concerns only the zones {\it I} and {\it II}, but the zones {\it III} and {\it IV} can be treated the same way.

In zone {\it III}, we have to deal with the so called ``photon sphere''. The estimate \eqref{a41} follows from a general polynomial bound of the resolvent in dimension $1$ (see \cite{BoMi04_01}).

In zone {\it IV}, the potentials $\ell ( \ell +1) V$ and $W$ are very small in comparison to $\lambda^{2}$. So they don't play any role, and we obtain the same estimate as in the free case $-\Delta$ (or as in the non trapping geometry).

\Subsection{Estimate close to $0$}  \label{a51}

This part is devoted to the proof of Proposition \ref{a50} and of part {\it i)} of Theorem \ref{a44}. Since $\chi (P - \lambda^{2} )^{-1} \chi$ has a meromorphic extension to $\C$, the number of resonance in $B (0 , R)$ is always bounded and point {\it i)} of Theorem \ref{a44} is clear. It is a classical result (see Theorem XIII.58 in \cite{ReSi78_01}) that $P_{\ell}$ has no eigenvalue in $\R \setminus \{ 0 \}$. On the other hand, from Proposition II.1 of the work of Bachelot and Motet-Bachelot \cite{BaMo93_01}, $0$ is not an eigenvalue of the operators $P_{\ell}$. Moreover, from the limiting absorption principle \cite{Mo81_01},
\begin{equation}
\Vert \< x \>^{- \alpha} (P_{\ell} - (z + i 0))^{-1} \< x \>^{- \alpha} \Vert < \infty ,
\end{equation}
for $z \in \R \setminus \{ 0 \}$ and any $\alpha > 1$, we know that $P_{\ell}$ has no resonance in $\R \setminus \{ 0 \}$.

We now study the resonance $0$ using a technique specific to the one dimensional case. We start with recalling some facts about outgoing Jost solutions. Let
\begin{equation} \label{a17}
Q = - \Delta + \widetilde{V}(x) ,
\end{equation}
be a Schr\"{o}dinger operator with $\widetilde{V} \in C^{\infty} (\R)$ decaying exponentially at infinity. For $\im \lambda >0$, there exists a unique couple of functions $e_{\pm} (x, \lambda )$ such that
\begin{equation*}
\left\{ \begin{aligned}
&(Q - \lambda^{2} ) e_{\pm} (x, \lambda) =0  \\
&\lim_{x \to \pm \infty} \big( e_{\pm} (x, \lambda ) - e^{\pm i \lambda x} \big) = 0
\end{aligned} \right.
\end{equation*}
The function $e_{\pm}$ is called the outgoing Jost solution at $\pm \infty$. Since $\widetilde{V} \in C^{\infty}(\R)$ decays exponentially at infinity, the functions $e_{\pm}$ can be extended, as a $C^{\infty}(\R)$ function in $x$, analytically in a strip $\{ \lambda \in \C; \ \im \lambda > - \varepsilon \}$, for some $\varepsilon >0$. Moreover, in such a strip, they satisfy
\begin{align}
\vert e_{\pm} (x, \lambda) - e^{\pm i \lambda x} \vert =& \CO ( e^{- x (\im \lambda + \delta )})  \text{ for } \pm x >0   \label{a16}  \\
\vert \partial_{x} e_{\pm} (x, \lambda) \mp i \lambda e^{\pm i \lambda x} \vert =& \CO ( e^{- x ( \im \lambda + \delta)})  \text{ for } \pm x >0  , \label{a18}
\end{align}
for some $\delta >0$. All these properties can be found in Theorem XI.57 of \cite{ReSi79_01}.

Using these Jost solutions, the kernel of $(Q - \lambda^{2})^{-1}$, for $\im \lambda >0$ takes the form
\begin{equation}  \label{a20}
R (x, y ,\lambda ) = \frac{1}{w(\lambda)} \big( e_{+} (x, \lambda ) e_{-} (y, \lambda ) H (x-y) + e_{-} (x, \lambda ) e_{+} (y, \lambda ) H (y-x) \big) ,
\end{equation}
where $H (x)$ is the Heaviside function
\begin{equation*}
H (x) = \left\{
\begin{aligned}
& 1 \text{ for } x >0  \\
& 0 \text{ for } x \leq 0
\end{aligned} \right. ,
\end{equation*}
and
\begin{equation}  \label{a15}
w ( \lambda ) = ( \partial_{x} e_{-} ) e_{+} - ( \partial_{x} e_{+} ) e_{-} ,
\end{equation}
is the wronskian between $e_{-}$ and $e_{+}$ (the right hand side of \eqref{a15} does not depend on $x$). In particular, $w (\lambda)$ is an analytic function on $\C$. Since the $e_{\pm}$ are always non-zero thanks to \eqref{a16}, the resonances are the zeros of $w (\lambda )$. Such a discussion can be found in the preprint of Tang--Zworski \cite{TaZw07_01}.

Remark that $P_{\ell}$ is of the form \eqref{a17}. If $0$ is a resonance of one of the $P_{\ell}$'s with $\ell \geq 1$, the Jost solutions $e_{\pm} (x, 0)$ are collinear. In particular, from \eqref{a16} and \eqref{a18}, the $C^{\infty}$ function $e_{+} (x,0)$ converge to two non zero limits at $\pm \infty$ and $\partial_{x} e_{+} (x,0)$ goes to $0$ as $x \to \pm \infty$. Since
\begin{equation}  \label{a19}
P_{\ell} = r^{-1} D_{x} r^{2} D_{x} r^{-1}  + \alpha^{2} r^{-2} \ell (\ell +1) ,
\end{equation}
we get, by an integration by parts,
\begin{align}
0 =& \int_{-R}^{R} ( P_{\ell} e_{+}) \overline{e_{+}} \, d x  \nonumber  \\
=& \ell (\ell +1) \int_{-R}^{R} \vert \alpha r^{-1} e_{+} \vert^{2} \, d x + \int_{-R}^{R} \vert r D_{x} ( r^{-1} e_{+} ) \vert^{2} \, d x - \Big[ i r^{-1} \overline{e_{+}} D_{x} (r^{-1} e_{+}) \Big]_{-R}^{R} . \label{b1}
\end{align}
Since $\partial_{x} (r^{-1} e_{+}) = r^{-1} \partial_{x} e_{+} - r^{-2} \alpha^{2} e_{+}$, the last term in \eqref{b1} goes to $0$ as $R$ goes to $+ \infty$. Thus, if $\ell \geq 1$, \eqref{b1} gives $e_{+} =0$ and $0$ is not a resonance of $P_{\ell}$.

We study now the case $\ell =0$. If $u \in C^{2}(\R)$ satisfies $P_{0} u =0$, we get from \eqref{a19}
\begin{equation*}
r^{2} D_{x} r^{-1} u = - i \beta,
\end{equation*}
where $\beta \in \C$ is a constant. Then
\begin{equation*}
u (x) = \alpha r (x) + \beta r(x) \int_{0}^{x} \frac{1}{r^{2} (t)} d t ,
\end{equation*}
where $\alpha , \beta \in \C$ are constants. Note that
\begin{equation*}
\widetilde{r} (x) := r(x) \int_{0}^{x} \frac{1}{r^{2} (t)} d t = \frac{x}{r_{\pm}} + \CO (1) ,
\end{equation*}
as $x \to \pm \infty$. Since $e_{\pm} (x,0)$ are $C^{\infty}$ functions bounded at $\pm \infty$ from \eqref{a16} which satisfy $P_{0} u =0$, the two functions $e_{\pm} (x,0)$ are collinear to $r$ and then $w (0) =0$ which means that $0$ is a resonance of $P_{0}$. The resolvent of $P_{0}$ has thus the form
\begin{equation*}
(P_{0} - \lambda^{2})^{-1} = \frac{\Pi_{J}}{\lambda^{J}} + \cdots + \frac{\Pi_{1}}{\lambda} + H ( \lambda ) ,
\end{equation*}
where $H ( \lambda )$ is an analytic family of bounded operators near $0$ and $\Pi_{J} \neq 0$.

For all $\lambda = i \varepsilon$ with $\varepsilon >0$, we have
\begin{equation*}
\Vert \lambda^{2} (P_{0} - \lambda^{2})^{-1} \Vert_{L^{2} \to L^{2}} = \Vert \varepsilon^{2} (P_{0} + \varepsilon^{2})^{-1} \Vert_{L^{2} \to L^{2}} \leq 1,
\end{equation*}
from the functional calculus. This inequality implies that $J \leq 2$ and
\begin{equation*}
\Vert \Pi_{2} \Vert_{L^{2} \to L^{2}} \leq 1.
\end{equation*}
If $f(x) \in L^{2}_{\text{loc}}$ is in the range of $\Pi_{2}$, we have $f \in L^{2}$ and $P_{0} f =0$. Then, $f \in H^{s}$ for all $s$ and $f$ is an eigenvector of $P_{0}$ for the eigenvalue $0$. This point is impossible because $P_{0}$ has no eigenvalue. Thus $\Pi_{2} =0$ and $J=1$.

So $w (\lambda )$ has a zero of order $1$ at $\lambda =0$. Since $e_{\pm} (x ,0) = r(x) /r_{\pm}$, \eqref{a20} implies that the kernel of $\Pi_{1}$ is given by
\begin{equation}
\Pi_{1} (x, y) = \frac{1}{w'(0) r_{+} r_{-}} r (x) r(y) = i \gamma r(x) r(y) .
\end{equation}
Finally, since $i \varepsilon (P_{0} + \varepsilon^{2})^{-1} \to \Pi_{1}$ as $\varepsilon \to 0$ and since $P_{0} + \varepsilon^{2}$ is a strictly positive operator, we get $\< -i \Pi_{1} u , u \> \geq 0$ for all $u \in L^{2}_{\text{comp}}$. In particular, $-i i \gamma >0$ and then $\gamma \in ] 0, + \infty [$.

\Subsection{Estimate for $\lambda$ small in comparison to $\ell$.}

In this section, we give an upper bound for the cut-off resolvent for $\lambda \in [ R , \ell / R] + i [ - C_{0} ,C_{0} ]$. We assume that $\lambda \in [N, 2 N] + i [-C_{0} , C_{0} ]$ with $N \in [R , \ell / R]$, and define a new semi-classical parameter $h = N^{-1}$, a new spectral parameter $z = h^{2} \lambda^{2} \in [1/4 , 4]+ i[- 4 C_{0} h, 4 C_{0} h]$ and
\begin{equation} \label{a37}
\widetilde{P} = - h^{2} \Delta + h^{2} \ell (\ell + 1) V (x) + h^{2} W (x) .
\end{equation}
With these notations, we have
\begin{equation} \label{a38}
( P_{\ell} - \lambda^{2} )^{-1} = h^{2} ( \widetilde{P} - z)^{-1} .
\end{equation}
We remark that $\beta^{2} := h^{2} \ell (\ell + 1) \gg 1$ in our window of parameters.  The potentials $V$ and $W$ have a holomorphic extension in a sector
\begin{equation} \label{a28}
\Sigma = \{ x \in \C; \ \vert \im x \vert \leq \theta_{0} \vert \re x \vert \text{ and } \vert \re x \vert \geq C \} ,
\end{equation}
for some $C, \theta_{0} >0$. From the form of $\alpha^{2}$ (see \eqref{a2}), there exist $\kappa_{\pm} >0$ and functions $f_{\pm} \in C^{\infty} ( \R^{\pm} ; [1/C ,C])$, $C >0$, analytic in $\Sigma$ such that
\begin{equation}  \label{a36}
V (x) = e^{\mp \kappa_{\pm} x} f_{\pm} (x),
\end{equation}
for $x \in \Sigma$ and $\pm \re x >0$. Moreover, $f_{\pm}$ have a (non zero) limit for $x \to \pm \infty$, $x \in \Sigma$.

Under these hypotheses, and following Proposition 4.4 of \cite{SaZw97_01}, we can use the specific estimate developed by Zworski in \cite{Zw99_01} for operators like \eqref{a37} with $V$ satisfying \eqref{a36}. In the beginning of Section 4 of \cite{Zw99_01}, Zworski defines a subtle contour $\Gamma_{\theta}$ briefly described in the following figure.
\begin{figure}[!h]
\begin{picture}(0,0)%
\includegraphics{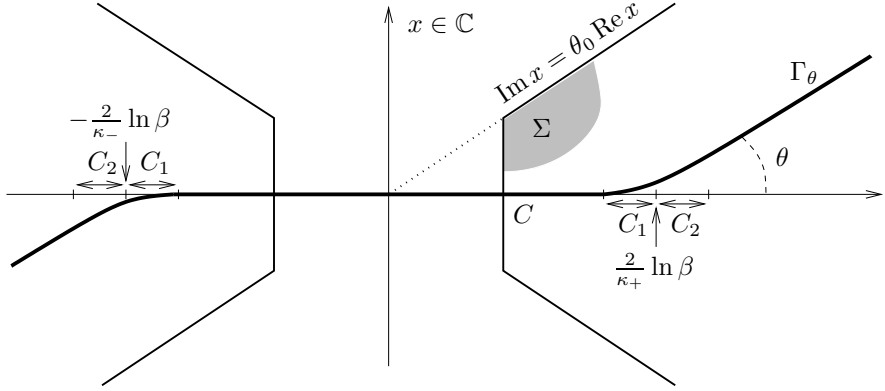}%
\end{picture}%
\setlength{\unitlength}{1579sp}%
\begingroup\makeatletter\ifx\SetFigFont\undefined%
\gdef\SetFigFont#1#2#3#4#5{%
  \reset@font\fontsize{#1}{#2pt}%
  \fontfamily{#3}\fontseries{#4}\fontshape{#5}%
  \selectfont}%
\fi\endgroup%
\begin{picture}(13844,6066)(-21,-5794)
\put(8251,-1861){\makebox(0,0)[lb]{\smash{{\SetFigFont{10}{12.0}{\rmdefault}{\mddefault}{\updefault}$\Sigma$}}}}
\put(9601,-3361){\makebox(0,0)[lb]{\smash{{\SetFigFont{10}{12.0}{\rmdefault}{\mddefault}{\updefault}$C_{1}$}}}}
\put(10426,-3361){\makebox(0,0)[lb]{\smash{{\SetFigFont{10}{12.0}{\rmdefault}{\mddefault}{\updefault}$C_{2}$}}}}
\put(2101,-2386){\makebox(0,0)[lb]{\smash{{\SetFigFont{10}{12.0}{\rmdefault}{\mddefault}{\updefault}$C_{1}$}}}}
\put(1276,-2386){\makebox(0,0)[lb]{\smash{{\SetFigFont{10}{12.0}{\rmdefault}{\mddefault}{\updefault}$C_{2}$}}}}
\put(976,-1711){\makebox(0,0)[lb]{\smash{{\SetFigFont{10}{12.0}{\rmdefault}{\mddefault}{\updefault}$- \frac{2}{\kappa_{-}} \ln \beta$}}}}
\put(9526,-4036){\makebox(0,0)[lb]{\smash{{\SetFigFont{10}{12.0}{\rmdefault}{\mddefault}{\updefault}$\frac{2}{\kappa_{+}} \ln \beta$}}}}
\put(6301,-211){\makebox(0,0)[lb]{\smash{{\SetFigFont{10}{12.0}{\rmdefault}{\mddefault}{\updefault}$x \in \C$}}}}
\put(7801,-1336){\rotatebox{33.7}{\makebox(0,0)[lb]{\smash{{\SetFigFont{10}{12.0}{\rmdefault}{\mddefault}{\updefault}$\im x = \theta_{0} \re x$}}}}}
\put(12076,-2311){\makebox(0,0)[lb]{\smash{{\SetFigFont{10}{12.0}{\rmdefault}{\mddefault}{\updefault}$\theta$}}}}
\put(7951,-3211){\makebox(0,0)[lb]{\smash{{\SetFigFont{10}{12.0}{\rmdefault}{\mddefault}{\updefault}$C$}}}}
\put(12301,-961){\makebox(0,0)[lb]{\smash{{\SetFigFont{10}{12.0}{\rmdefault}{\mddefault}{\updefault}$\Gamma_{\theta}$}}}}
\end{picture}%
\caption{The set $\Sigma$ and the contour $\Gamma_{\theta}$.}
\label{gamma}
\end{figure}

Recall that the distorted operator $\widetilde{P}_{\theta} = \widetilde{P}_{\vert_{\Gamma_{\theta}}}$ is defined by
\begin{equation}
\widetilde{P}_{\theta} u = (\widetilde{P} u)_{\vert_{\Gamma_{\theta}}}
\end{equation}
for all $u$ analytic in $\Sigma$ and then extended as a differential operator on $L^{2} (\Gamma_{\theta})$ by means of almost analytic functions. The resonances of $\widetilde{P}$ in the sector $S_{\theta} = \{ e^{-2 i s} r ; \ 0< s <\theta \text{ and } r \in ]0, + \infty [ \} = e^{2 i ]- \theta ,0]} ]0, + \infty [$ are then the eigenvalues of $\widetilde{P}_{\theta}$ in that set. For the general theory of resonances, see the paper of Sj\"{o}strand \cite{Sj97_01} or his book \cite{Sj07_01}.

For $\theta$ large enough, Proposition 4.1 of \cite{Zw99_01} proves that $\widetilde{P}$ has no resonances in $[1/4 , 4]+ i[- 4 C_{0} h, 4 C_{0} h]$. Moreover, for $z$ in that set, this proposition gives the uniform estimate
\begin{equation}
\Vert ( \widetilde{P}_{\theta} -z)^{-1} \Vert \leq C.
\end{equation}
Since $\Gamma_{\theta}$ coincides with $\R$ for $x \in \supp \chi$, we have
\begin{equation}
\chi ( \widetilde{P}- z)^{-1} \chi = \chi ( \widetilde{P}_{\theta} - z)^{-1} \chi ,
\end{equation}
from Lemma 3.5 of \cite{SjZw91_01}. Using \eqref{a37}, we immediately obtain
\begin{equation}
\Vert \chi ( P_{\ell} - \lambda^{2} )^{-1} \chi \Vert \leq \frac{C}{\< \lambda \>^{2}} ,
\end{equation}
which is exactly \eqref{a40}.

\Subsection{Estimate for $\lambda$ of order $\ell$.} \label{a65}

In this part, we study the cut-off resolvent for the energy $\lambda \in [ \ell / R , R \ell ] + i [ - C_{0} , C_{0} ]$. In that zone, we have to deal with the photon sphere. We define the new semi-classical parameter $h = (\ell (\ell +1))^{-1/2}$ and
\begin{equation}
\widetilde{P} = - h^{2} \Delta + V (x) + h^{2} W (x) .
\end{equation}
As previously, we have
\begin{equation}  \label{a33}
( P_{\ell} - \lambda^{2} )^{-1} = h^{2} ( \widetilde{P} - z)^{-1} ,
\end{equation}
where
\begin{equation}  \label{a34}
z = h^{2} \lambda^{2} \in [1/2 R^{2} , R^{2} ] + i [- 3 R C_{0} h ,0] = [a ,b] + i [-c , c] ,
\end{equation}
with $0< a < b$ and $0<c$. Note that $V$ is of the form:
\begin{figure}[!h]
\begin{picture}(0,0)%
\includegraphics{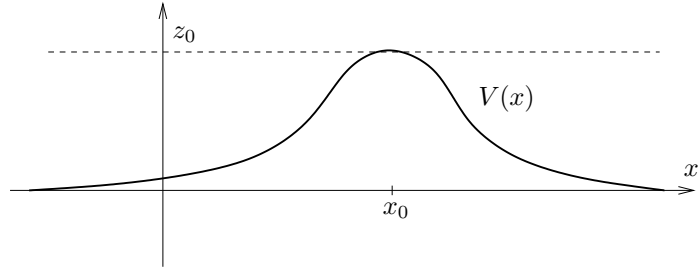}%
\end{picture}%
\setlength{\unitlength}{1579sp}%
\begingroup\makeatletter\ifx\SetFigFont\undefined%
\gdef\SetFigFont#1#2#3#4#5{%
  \reset@font\fontsize{#1}{#2pt}%
  \fontfamily{#3}\fontseries{#4}\fontshape{#5}%
  \selectfont}%
\fi\endgroup%
\begin{picture}(10844,4244)(1179,-3983)
\put(3751,-361){\makebox(0,0)[lb]{\smash{{\SetFigFont{10}{12.0}{\rmdefault}{\mddefault}{\updefault}$z_{0}$}}}}
\put(8551,-1411){\makebox(0,0)[lb]{\smash{{\SetFigFont{10}{12.0}{\rmdefault}{\mddefault}{\updefault}$V(x)$}}}}
\put(7051,-3136){\makebox(0,0)[lb]{\smash{{\SetFigFont{10}{12.0}{\rmdefault}{\mddefault}{\updefault}$x_{0}$}}}}
\put(11776,-2536){\makebox(0,0)[lb]{\smash{{\SetFigFont{10}{12.0}{\rmdefault}{\mddefault}{\updefault}$x$}}}}
\end{picture}%
\caption{The potential $V (x)$.}
\label{f1}
\end{figure}\newline
In particular, $V$ has a non-degenerate maximum at energy $z_{0} >0$. On the other hand, for $z \neq z_{0}$, $z >0$, the energy level $z$ is non trapping for $\widetilde{p}_{0} (x, \xi ) = \xi^{2} + V (x)$, the principal semi-classical symbol of $\widetilde{P}$. We defined $\widetilde{P}_{\theta}$ by standard distortion (see Sj\"{o}strand \cite{Sj97_01}) and can apply the following general upper bound on the cut-off resolvent in dimension one.

\begin{lemma}[{\bf Lemma 6.5 of \cite{BoMi04_01}}]\sl  \label{a29}
We assume that $n=1$ and that the critical points of $p_{0} (x, \xi )$ on the energy level $E_{0}$ are non-degenerate (i.e. the points $(x , \xi ) \in \widetilde{p}_{0} ( \{ E_{0} \} )$ such that $\nabla \widetilde{p}_{0} (x, \xi ) =0$ satisfy $\Hess \widetilde{p}_{0} (x, \xi )$ is invertible). Then, there exists $\varepsilon >0$ such that, for $E \in [E_{0} - \varepsilon , E_{0} + \varepsilon ]$ and $\theta = N h$ with $N>0$ large enough,
\begin{equation}  \label{a30}
\Vert (\widetilde{P}_{\theta} - z)^{-1} \Vert = \CO (h^{-M}) \prod_{\fract{z \in {\rm Res} \, \widetilde{P}}{\vert z -z_{j} \vert < \varepsilon \theta}} \frac{h}{\vert z - z_{j} \vert}
\end{equation}
for $\vert z -E \vert < \varepsilon \theta /2$ and some $M>0$ which depends on $N$.
\end{lemma}

\noindent
Note that there is a slight error in the statement of the lemma in \cite{BoMi04_01}. Indeed, $M$ depends on $N$, and in the proof of this lemma, the right hand side of (6.18), $\CO ( \ln (1/\theta))$, must be replaced by $\CO (\theta h^{-1} \ln ( 1/ \theta))$.

Recall that, from Proposition 4.3 \cite{SaZw97_01}, which is close to the work of Sj\"{o}strand \cite{Sj87_01} on the the resonances associated to a critical point, there exists an injective map $b(h)$ from
\begin{equation}  \label{a53}
\Gamma_{0} (h) = \big\{ \mu_{j} = z_{0} - i h \sqrt{\vert V''( x_{0} ) \vert /2} ( j +1/2) ; \ j \in \N_{0} \big\} ,
\end{equation}
into the set of resonances of $\widetilde{P}$ such that
\begin{equation} \label{a54}
b(h)(\mu ) - \mu = o (h), \ \mu \in \Gamma_{0}(h) ,
\end{equation}
and such that all the resonances in $[a/2 , 2b] + i [- c , c ]$ are in the image of $b(h)$. In particular, the number of resonances of $\widetilde{P}$ is bounded in $[a/2 , 2b] + i [- c , c]$. Furthermore, the operator $\widetilde{P}$ has no resonance in
\begin{equation*}
\Omega (h) =[a /2 , 2b] + i [- ( \mu_{0} - \varepsilon h) , ( \mu_{0} - \varepsilon h)] ,
\end{equation*}
for any $\varepsilon >0$ and $h$ small enough.

Using a compactness argument, we get \eqref{a30} for all $z \in [a ,b ] + i [- c , c]$. Thus, from \eqref{a33}, \eqref{a34}, $\chi (\widetilde{P} - z)^{-1} \chi = \chi (\widetilde{P}_{\theta} - z)^{-1} \chi$, the estimate $\< \lambda \> \lesssim h^{-1} = \sqrt{\ell ( \ell +1)} \lesssim \< \lambda \>$ for $\lambda \in [ \ell / R , R \ell ] + i [ - C_{0} , 0]$, Lemma \ref{a29} and the previous discussion, we get
\begin{equation}
\Vert \chi ( P_{\ell} - \lambda^{2} )^{-1} \chi \Vert \leq C \< \lambda \>^{C} \prod_{\fract{z_{j} \in {\rm Res} \, P}{\vert \lambda - \lambda_{j} \vert < 1}} \frac{1}{\vert \lambda - \lambda_{j} \vert} ,
\end{equation}
for $\lambda \in [ \ell / R , R \ell ] + i [ - C_{0} , C_{0}]$ and \eqref{a41} follows.

On the other hand, $\widetilde{P}$ has no resonance in $\Omega (h)$ and in that set
\begin{equation}  \label{a31}
\Vert \chi (\widetilde{P} - z)^{-1} \chi \Vert \lesssim \left\{
\begin{aligned}
& h^{-M} &&\text{on } \Omega (h) , \\
&\frac{1}{\vert \im z \vert} &&\text{on } \Omega (h) \cap \{ \im z > 0 \} .
\end{aligned} \right.
\end{equation}
We can now applied the following version of the so-called ``semi-classical maximum principle'' introduced by Tang--Zworski \cite{TaZw98_01}.

\begin{lemma}[{\bf Burq}]\sl \label{a35}
Suppose that $f(z,h)$ is a family of holomorphic functions defined for $0<h<1$ in a neighborhood of
\begin{equation*}
\Omega (h) = [a/2 , 2b] + i [- c h , c h] ,
\end{equation*}
with $0<a<b$ and $0<c$, such that
\begin{equation*}
\vert f (z,h) \vert \lesssim \left\{
\begin{aligned}
& h^{-M} &&\text{on } \Omega (h) , \\
&\frac{1}{\vert \im z \vert} &&\text{on } \Omega (h ) \cap \{ \im z > 0 \} .
\end{aligned} \right.
\end{equation*}
Then, there exists $h_{0}, C > 0$ such that, for any $0<h<h_{0}$,
\begin{equation}
\vert f(z,h) \vert \leq C \frac{\vert \ln h \vert}{h} e^{C \vert \im z \vert \vert \ln h \vert /h} ,
\end{equation}
for $z \in [a,b]+i [-c h, 0]$.
\end{lemma}

\noindent
This lemma is strictly analogous to Lemma 4.7 of the paper of Burq \cite{Bu04_01}. Combining \eqref{a33}, \eqref{a34}, $\< \lambda \> \lesssim h^{-1} \lesssim \< \lambda \>$ with this lemma, we obtain
\begin{equation}
\Vert \chi (P_{\ell} - \lambda^{2})^{-1} \chi \Vert \leq C \frac{\ln \< \lambda \>}{\< \lambda \>} e^{C \vert \im \lambda \vert \ln \< \lambda \>} ,
\end{equation}
for $\lambda \in [ \ell / R , R \ell ] + i [ - \varepsilon , 0]$, for some $\varepsilon >0$.

\Subsection{Estimate for the very large values of $\lambda$.}

Here, we study the resolvent for $\vert \lambda \vert \gg \ell$. More precisely, we assume that
\begin{equation*}
\lambda \in [ N , 2N] + i [- C \ln N , C_{0}] ,
\end{equation*}
for some $C>0$ fixed and $N \gg \ell$. We define the new semi-classical parameter $h = N^{-1}$ and
\begin{equation*}
z = h^{2} \lambda^{2} \in h^{2} [N^{2} /2 , 4 N^{2}] + i h^{2} [ - 4 C N \ln N ,4 C_{0} N^{-1}] \subset [ a , b ] + i [ - c h \vert \ln h \vert , c h] ,
\end{equation*}
for some $0<a<b$ and $0<c$. Then, $P_{\ell}$ can be written
\begin{equation*}
P_{\ell} - \lambda^{2} =  h^{-2} ( \widetilde{P} - z ),
\end{equation*}
where
\begin{equation*}
\widetilde{P} = - h^{2} \Delta + \mu V (x) + \nu W (x) ,
\end{equation*}
with $\mu = \ell (\ell +1) h^{2}$, $\nu = h^{2}$. For $N \gg \ell$, the coefficients $\mu , \nu$ are small, and the operator $\widetilde{P}$ is uniformly non trapping for $z \in [a,b]$. We can expect a uniform bound of the cut-off resolvent in $[ a , b ] + i [ - c h \vert \ln h \vert ,0]$. Such a result is proved in the following lemma.

\begin{lemma}\sl  \label{a21}
For all $\chi \in C^{\infty}_{0} (\R)$, there exist $\mu_{0} , \nu_{0} , h_{0}, C>0$ such that, for all $\mu < \mu_{0}$, $\nu < \nu_{0}$ and $h < h_{0}$, $\widetilde{P}$ has no resonance in $[ a , b ] + i [ c h \ln h , c h]$. Moreover
\begin{equation}  \label{a22}
\Vert \chi (\widetilde{P} - z)^{-1} \chi \Vert \leq \frac{C}{h} e^{C \vert \im z \vert / h} ,
\end{equation}
for all $z \in [ a , b ] + i [ - c h \vert \ln h \vert , c h]$.
\end{lemma}

\noindent
Assume first Lemma \ref{a21}. For $\lambda \in [ N , 2N] + i [- C \ln N ,C_{0}]$, we have
\begin{align*}
\Vert \chi (P_{\ell} - \lambda^{2})^{-1} \chi \Vert =& \Vert h^{2} \chi ( \widetilde{P} - z )^{-1} \chi \Vert \\
\leq& C h e^{C \vert \im z \vert / h}  \\
\leq& \frac{C}{\vert \lambda \vert} e^{4 C \vert \im \lambda \vert} ,
\end{align*}
and the estimate \eqref{a43} follows.

\begin{proof}[Proof of Lemma \ref{a21}]
For $\mu$ and $\nu$ small and fixed, the estimate \eqref{a22} is already know. The proof can be found in the book of Va{\u\i}nberg \cite{Va89_01} in the classical case and in the paper of Nakamura--Stefanov--Zworski \cite{NaStZw03_01} in our semi-classical setting. To obtain Lemma \ref{a21}, we only have to check the uniformity (with respect to $\mu$ and $\nu$) in the proof of \cite[Proposition 3.1]{NaStZw03_01}.

\noindent
{\it $\bullet$ Limiting absorption principle.} \newline
The point is to note that
\begin{equation}
A = x h D_{x} + h D_{x} x ,
\end{equation}
is a conjugate operator for all $\mu , \nu \ll 1$. Let $g \in C^{\infty}_{0} ([a/3, 3b] ; [0,1])$ be equal to $1$ near $[a/2,2b]$. The operator $g(\widetilde{P}) A g(\widetilde{P})$ is well defined on $D (A)$, and its closure, $\CA$, is self-adjoint. The operator $\widetilde{P}$ is of class $C^{2} ( \CA )$. Recall that $\widetilde{P}$ is of class $C^{r} ( \CA )$ if there exists $z \in \C \setminus \sigma ( \widetilde{P} )$ such that
\begin{equation*}
\R \ni t \to e^{it \CA} (\widetilde{P}  - z)^{-1} e^{-it \CA} ,
\end{equation*}
is $C^{r}$ for the strong topology of $L^{2}$ (see \cite[Section 6.2]{AmBoGe96_01} for more details).

We have
\begin{equation}
i h^{-1} [ \widetilde{P} , A ] = 4 \widetilde{P} - 4 \mu V - 4 \nu W - 2 \mu x V' - 2 \nu x W' .
\end{equation}
In particular, for $\mu$ and $\nu$ small enough, we easily obtain
\begin{equation}
1_{[a/2 , 2b]} (\widetilde{P}) i [ \widetilde{P} , \CA ] 1_{[a/2 , 2b]} (\widetilde{P}) \geq a h 1_{[a/2 , 2b]} (\widetilde{P}) .  \label{a26}
\end{equation}
Note that this Mourre estimate is uniform with respect to $\mu , \nu$.

It is also easy to check that
\begin{equation}   \label{a27}
\begin{gathered}
\Vert \< x \>^{-1} \CA \Vert \leq C \\
\Vert (\widetilde{P} +i)^{-1} [\widetilde{P} , A] \Vert \leq C h    \\
\Vert (\widetilde{P} +i)^{-1} [ [\widetilde{P} , A] ,A] \Vert \leq C h^{2} \\
\Vert (\widetilde{P} +i)^{-1} [\widetilde{P} , [ \widetilde{P}  ,A]] \Vert \leq C h^{2} \\
\Vert (\widetilde{P} +i)^{-1} A [\widetilde{P} , [ \widetilde{P}  ,A]] \Vert \leq C h^{2} ,
\end{gathered}
\end{equation}
uniformly in $\mu , \nu$.

The regularity $\widetilde{P} \in C^{2} ( \CA )$, the Mourre estimate \eqref{a26} and the upper bound \eqref{a27} are the key assumptions for the limiting absorption principle. In particular, from, by example, the proof of Proposition 3.2 in \cite{AlBoRa07_01} which is an adaptation of the theorem of Mourre \cite{Mo81_01}, we obtain the following estimate: For $\alpha >1/2$, there exist $\mu_{0} , \nu_{0} , h_{0} ,C >0$, such that
\begin{equation}
\Vert \< x \>^{- \alpha} (\widetilde{P} - z)^{-1} \< x \>^{- \alpha} \Vert \leq C h^{-1} ,
\end{equation}
for all $\mu < \mu_{0}$, $\nu < \nu_{0}$, $h < h_{0}$ and $z \in [a/2 , 2b] + i ] 0 , c h ]$. In particular,
\begin{equation}  \label{a25}
\Vert \chi (\widetilde{P} - z)^{-1} \chi \Vert \leq C h^{-1} .
\end{equation}
for $z \in [a/2 , 2b] + i [ 0 , c h ]$.

\noindent
{\it $\bullet$ Polynomial estimate in the complex.} \newline
The second point of the proof is to obtain a polynomial bound of the distorted resolvent. To obtain such bounds, we use the paper of Martinez \cite{Ma02_01}. In this article, the author studies the resonance of $Q = - h \Delta + \widetilde{V} (x)$ where $\widetilde{V}$ is a $C^{\infty} ( \R^{n})$ function which can be extended analytically in a domain like $\Sigma$ (see \eqref{a28}) and decays in this domain. If the energy level $z_{0}$ is non trapped for the symbol $q (x, \xi ) = \xi^{2} + \widetilde{V} (x)$, the operator $Q$ has no resonance in $[ z_{0} - \delta , z_{0} + \delta ] + i [ A h \ln h ,0]$ for a $\delta$ small enough and any $A >0$. Moreover,
\begin{equation}  \label{a23}
\Vert (Q_{\theta} - z)^{-1} \Vert \leq C h^{-C}
\end{equation}
for $z \in [ z_{0} - \delta , z_{0} + \delta ] + i [ A h \ln h ,0]$. Here $Q_{\theta}$ denotes the distorted operator outside of a large ball of angle $\theta = B h \vert \ln h \vert$, with $B \gg A$.

Of course, $\widetilde{P}$ satisfies the previous assumption on $Q$, for $\mu$ and $\nu$ fixed small enough. But, following the proof of \eqref{a23} in \cite[Section 4]{Ma02_01} from line to line, one can prove that \eqref{a23} is uniformly true for $\mu , \nu \ll 1$. This means that there exist $\mu_{0} , \nu_{0} , h_{0} ,C >0$ such that
\begin{equation} \label{a24}
\Vert \chi (\widetilde{P} -z)^{-1} \chi \Vert = \Vert \chi (\widetilde{P}_{\theta} -z)^{-1} \chi \Vert \leq C h^{-C} ,
\end{equation}
for all $\mu < \mu_{0}$, $\nu < \nu_{0}$, $h < h_{0}$ and $z \in [ a/2 , 2b ] + i [ c h \ln h ,0]$.

\noindent
{\it $\bullet$ Semi-classical maximum principle.} \newline
To finish the proof, we use a version of the semi-classical maximum principle. This argument can be found in \cite[Proposition 3.1]{NaStZw03_01}, but we give it for the convenience of the reader.

We can construct a holomorphic function $f (z,h)$ with the following properties:
\begin{align*}
&\vert f \vert \leq C \quad \text{ for } z \in [ a/2 , 2b ] + i [ c h \ln h ,0] , \\
&\vert f \vert \geq 1 \quad \text{ for } z \in [ a , b ] + i [ c h \ln h ,0] ,  \\
&\vert f \vert \leq h^{M} \quad \text{ for } z \in [ a/2 , 2b ] \setminus [ 2a/3 , 3b/2 ] + i [ c h \ln h ,0] ,
\end{align*}
where $M$ is the constant $C$ given in \eqref{a24}. We can then apply the maximum principle in $[ a/2 , 2b ] + i [ c h \ln h ,0]$ to the subharmonic function
\begin{equation*}
\ln \Vert \chi (\widetilde{P} -z)^{-1} \chi \Vert + \ln \vert f (z,h) \vert - C \frac{\im z}{h} ,
\end{equation*}
proving the lemma with \eqref{a25} and \eqref{a24}.
\end{proof}

\section{Proof of the main theorem}

\Subsection{Resolvent estimates for $L_{\ell}$.}

The cut-off resolvent estimates for $P_{\ell}$ give immediately cut-off resolvent estimates for $L_{\ell}$.

\begin{proposition}\sl  \label{prop2.4}
Let $\chi\in C_0^{\infty}(\R)$. Then the operator $\chi(L_{\ell}-\lambda)^{-1}\chi$ sends ${\mathcal E}_{\ell}^{\rm mod}$ into itself and we have uniformly in $\ell$:
\begin{equation}  \label{REL}
\Vert\chi(L_{\ell}-z)^{-1}\chi\Vert_{{\mathcal L}({\mathcal E}_{\ell}^{\rm mod})} \lesssim \langle z \rangle \Vert \chi(P_{\ell} - z^2)^{-1} \chi \Vert
\end{equation}
\end{proposition}

\begin{proof}
Using Theorem \ref{a44}, \eqref{RL}, the equivalence of the norms ${\mathcal E}_{a,b}^{\rm mod}$ as well as the fact that we can always replace $u$ by $\widetilde{\chi}u,\, \widetilde{\chi}\in C_0^{\infty}(\R),\, \widetilde{\chi}\chi=\chi$ we see that it is sufficient to show:
\begin{align}
\label{2.4.1}
\Vert\chi(P_{\ell}-z^2)^{-1}\chi u \Vert_{H^1} &\lesssim \Vert\widetilde{\chi}(P_{\ell}-z^2)^{-1}\widetilde{\chi}\Vert \Vert u\Vert_{H^1},\\
\label{2.4.2}
\Vert\chi(P_{\ell}-z^2)^{-1}\chi u \Vert_{H^1} &\lesssim \langle z \rangle\Vert\widetilde{\chi}(P_{\ell}-z^2)^{-1}\widetilde{\chi}\Vert \Vert u\Vert_{L^2},\\
\label{2.4.3}
\Vert\chi(P_{\ell}-z^2)^{-1}P_{\ell}\chi u \Vert_{L^2} &\lesssim \langle z \rangle\Vert\widetilde{\chi}(P_{\ell}-z^2)^{-1}\widetilde{\chi}\Vert \Vert u\Vert_{H^1}.
\end{align}
Using complex interpolation we see that it is sufficient to show:
\begin{align}
\label{2.4.4}
\Vert\chi(P_{\ell}-z^2)^{-1}\chi u \Vert_{H^2} \lesssim& \Vert\widetilde{\chi}(P_{\ell}-z^2)^{-1}\widetilde{\chi}\Vert \Vert u\Vert_{H^2},\\
\label{2.4.5}
\Vert\chi(P_{\ell}-z^2)^{-1}\chi u \Vert_{H^2} \lesssim& \langle z\rangle^2\Vert\widetilde{\chi}(P_{\ell}-z^2)^{-1}\widetilde{\chi}\Vert \Vert u\Vert_{L^2},\\
\label{2.4.6}
\Vert\chi(P_{\ell}-z^2)^{-1}P_{\ell}\chi u \Vert_{L^2} \lesssim& \Vert\widetilde{\chi}(P_{\ell}-z^2)^{-1}\widetilde{\chi}\Vert \Vert u\Vert_{H^2}.
\end{align}
We start with (\ref{2.4.6}) which follows from 
\begin{eqnarray*}
\chi (P_{\ell} - z^2)^{-1} P_{\ell}\chi=\chi (P_{\ell} - z^2)^{-1} \chi P_{\ell}+\chi (P_{\ell} - z^2)^{-1} {[}P_{\ell},\chi{]}u.
\end{eqnarray*}
Let us now observe that
\begin{align*}
P_{\ell}\chi (P_{\ell} - z^2)^{-1} \chi u =& {[}P_{\ell},\chi{]} (P_{\ell} - z^2)^{-1} \chi+\chi (P_{\ell} - z^2)^{-1} P_{\ell}\chi u\\
=&\widetilde{\chi}(P_{\ell}+i)^{-1}{[}P_{\ell},{[}P_{\ell},\chi{]}{]} (P_{\ell} - z^2)^{-1} \chi u\\
&+\widetilde{\chi}(P_{\ell}+i)^{-1}{[}P_{\ell},\chi{]} (P_{\ell} - z^2)^{-1} (P_{\ell}+i)\chi u+\chi (P_{\ell} - z^2)^{-1} P_{\ell}\chi u.
\end{align*}
From this identity we obtain (\ref{2.4.4}) and (\ref{2.4.5}) using (\ref{2.4.6}) (for (\ref{2.4.4})) and that $(P_{\ell}+i)^{-1}{[}P_{\ell},{[}P_{\ell},\chi{]}{]}$ is uniformly bounded.
\end{proof}

\Subsection{Resonance expansion for the wave equation.}

For the proof of the main theorem we follow closely the ideas of Va{\u\i}nberg \cite[Chapter X.3]{Va89_01}. If ${\mathcal N}$ is a Hilbert space we will note $L^2_{\nu}(\R; {\mathcal N} )$ the space of all functions $v(t)$ with values in ${\mathcal N}$ such that $e^{-\nu t}v(t)\in L^2(\R; {\mathcal N} )$.
Let $u\in {\mathcal E}^{\rm mod}_{\ell}$ and
\begin{eqnarray*}
v(t)=\left\{\begin{array}{cc} e^{-itL_{\ell}}u& t\ge0,\\ 0 & t<0. \end{array}\right.
\end{eqnarray*}
Then $v\in L^2_{\nu}(\R;{\mathcal E}_{\ell})$ for all $\nu>0$. We can define 
\begin{eqnarray*}
\tilde{v}(k)=\int_0^{\infty}v(t)e^{ikt}dt
\end{eqnarray*}
as an element of ${\mathcal E}$ for all $k$ with $\im k>0$. The function $\tilde{v}$ depends analytically on $k$ when $\im k>0$. Also, on the line $\im k=\nu$ the function belongs to $L^2(\R; {\mathcal E}_{\ell})$. We have the inversion formula:
\begin{eqnarray*}
v(t)=\frac{1}{2\pi}\int_{-\infty+i\nu}^{\infty+i\nu}e^{-ik t}\tilde{v}(k) \, d k
\end{eqnarray*}
and the integral converges in $L^2_{\nu}(\R;{\mathcal E}_{\ell})$ for all $\nu>0$.
From the functional calculus we know that 
\begin{eqnarray*}
\tilde{v}(k)= - i (L_{\ell} -k)^{-1}u
\end{eqnarray*}
for all $k$ with $\im k>0$. We therefore obtain for all $t\ge0$:
\begin{eqnarray}
\label{3.3}
e^{-itL}u=\frac{1}{2 \pi i} \int_{-\infty+i\nu}^{\infty+i\nu}(L_{\ell} -k)^{-1}e^{-ikt}u \, d k,
\end{eqnarray}
where the integral is convergent in $L^2_{\nu}(\R;{\mathcal E}_{\ell})$. In the following, we denote $\widehat{R}_{\chi}^{\ell} ( k)$ the meromorphic extension of $\chi (L_{\ell} -k)^{-1} \chi$.

\begin{lemma}\sl  \label{a61}
Let $\chi \in C^{\infty}_{0} ( \R)$, $N\geq 0$. Then, there exist bounded operators $B_{j} \in {\mathcal L}({\mathcal E}_{\ell}^{{\rm mod} ,-q};$ ${\mathcal E}_{\ell}^{{\rm mod} ,-j-q})$, $j = 0 , \ldots ,N$, $q \in \N_{0}$ and $B \in {\mathcal L}({\mathcal E}_{\ell}^{{\rm mod} ,-q} ;{\mathcal E}_{\ell}^{{\rm mod}, -N -1 -q})$, $q \in \N_{0}$ such that
\begin{equation}
\widehat{R}_{\chi}^{\ell} ( k) = \sum_{j =0}^{N} \frac{1}{( k - i ( \nu +1))^{j+1}} B_{j} + \frac{1}{( k - i ( \nu +1))^{N +1}} B \widehat{R}_{\widetilde{\chi}}^{\ell} (k) \chi ,
\end{equation}
for some $\widetilde{\chi} \in C^{\infty}_{0} (\R)$ with $\chi \widetilde{\chi} = \chi$.
\end{lemma}

\begin{proof}
We proceed by induction over $N$. For $N=0$, we write
\begin{equation*}
(L_{\ell} - k)^{-1} + \frac{1}{k-i(\nu+1)} = \frac{1}{k-i(\nu+1)}(L_{\ell} - i(\nu +1))(L_{\ell} -k)^{-1}.
\end{equation*}
and choose $B_{0} = - \chi^{2}$. Then
\begin{equation}  \label{a60}
\widehat{R}_{\chi}^{\ell} ( k) - \frac{1}{k-i(\nu+1)} B_{0} = \frac{1}{k-i(\nu+1)} \widetilde{B}_{\chi , \widetilde{\chi}} \widehat{R}_{\widetilde{\chi}}^{\ell} (k) \chi ,
\end{equation}
where $\widetilde{B}_{\chi , \widetilde{\chi}} = \chi (L_{\ell} - i(\nu +1)) \widetilde{\chi}$, with $\chi = \chi \widetilde{\chi}$, is in the space ${\mathcal L}({\mathcal E}_{\ell}^{m, -q} ;{\mathcal E}_{\ell}^{m,-1-q})$.

Let us suppose that the lemma is proved for $N \geq 0$. We put
\begin{equation}
B_{N+1} = \frac{1}{(k - i ( \nu +1))^{N+1}} B \widetilde{\chi}^{2} \chi
\end{equation}
Using \eqref{a60}, we get
\begin{align}
\widehat{R}_{\chi}^{\ell} ( k) =& \sum_{j =0}^{N} \frac{1}{( k - i ( \nu +1))^{j+1}} B_{j} + \frac{1}{( k - i ( \nu +1))^{N +1}} B \widehat{R}_{\widetilde{\chi}}^{\ell} \chi  \nonumber \\
=& \sum_{j =0}^{N+1} \frac{1}{( k - i ( \nu +1))^{j+1}} B_{j} + \frac{1}{( k - i ( \nu +1))^{N +2}} B  \widetilde{B}_{\widetilde{\chi} , \widetilde{\widetilde{\chi}}}  \widehat{R}_{\widetilde{\widetilde{\chi}}}^{\ell} \chi ,
\end{align}
with $\widetilde{\widetilde{\chi}} \in C^{\infty}_{0} (\R)$ with $\widetilde{\widetilde{\chi}} \widetilde{\chi} = \widetilde{\chi}$. This proves the lemma.
\end{proof}

Let us define 
\begin{equation*}
\widetilde{R}_{\chi}^{\ell} (k) = \widehat{R}_{\chi}^{\ell} ( k) - \sum_{j =0}^{1} \frac{1}{( k - i ( \nu +1))^{j+1}} B_{j}.
\end{equation*}
Then, Lemma \ref{a61} implies
\begin{eqnarray}
\label{3.4}
\Vert \widetilde{R}_{\chi}^{\ell} (k) \Vert_{{\mathcal L}({\mathcal E}^{\rm mod}_{\ell};{\mathcal E}_{\ell}^{{\rm mod},-2})} \lesssim \frac{1}{\langle k \rangle^{2}} \Vert \widehat{R}_{\chi}^{\ell} (k) \Vert_{{\mathcal L}({\mathcal E}_{\ell}^{\rm mod} ;{\mathcal E}_{\ell}^{\rm mod})}.
\end{eqnarray}
Now observe that 
\begin{equation}
\int_{-\infty+i\nu}^{\infty+i\nu}\frac{B_{j}}{(k-i(\nu+1) )^{-j-1}}e^{-ikt}dk=0.
\end{equation}
Therefore (\ref{3.3}) becomes:
\begin{eqnarray*}
\chi e^{-i t L} \chi u= \frac{1}{2 \pi i} \int_{-\infty+i\nu}^{\infty+i\nu} \widetilde{R}_{\chi}^{\ell} (k) e^{-ikt} u \, d k ,
\end{eqnarray*}
where the previous integral is absolutely convergent in ${\mathcal L}({\mathcal E}^{\rm mod}_{\ell};{\mathcal E}_{\ell}^{{\rm mod},-2})$.

We first show part $(i)$ of the theorem. Integrating along the path
indicated in Figure \ref{paf} we obtain by the Cauchy theorem:
\begin{figure}[!h]
\begin{center}
\begin{picture}(0,0)%
\includegraphics{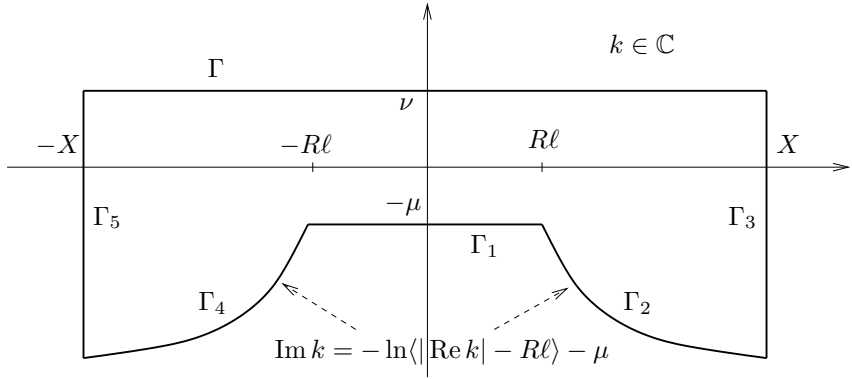}%
\end{picture}%
\setlength{\unitlength}{1579sp}%
\begingroup\makeatletter\ifx\SetFigFont\undefined%
\gdef\SetFigFont#1#2#3#4#5{%
  \reset@font\fontsize{#1}{#2pt}%
  \fontfamily{#3}\fontseries{#4}\fontshape{#5}%
  \selectfont}%
\fi\endgroup%
\begin{picture}(13319,5969)(-621,-4883)
\put(3601,-4561){\makebox(0,0)[lb]{\smash{{\SetFigFont{10}{12.0}{\rmdefault}{\mddefault}{\updefault}$\im k = - \ln \< \vert \re k \vert - R \ell \> - \mu$}}}}
\put(5551,-661){\makebox(0,0)[lb]{\smash{{\SetFigFont{10}{12.0}{\rmdefault}{\mddefault}{\updefault}$\nu$}}}}
\put(5326,-2236){\makebox(0,0)[lb]{\smash{{\SetFigFont{10}{12.0}{\rmdefault}{\mddefault}{\updefault}$- \mu$}}}}
\put(10726,-2461){\makebox(0,0)[lb]{\smash{{\SetFigFont{10}{12.0}{\rmdefault}{\mddefault}{\updefault}$\Gamma_{3}$}}}}
\put(751,-2461){\makebox(0,0)[lb]{\smash{{\SetFigFont{10}{12.0}{\rmdefault}{\mddefault}{\updefault}$\Gamma_{5}$}}}}
\put(6676,-2911){\makebox(0,0)[lb]{\smash{{\SetFigFont{10}{12.0}{\rmdefault}{\mddefault}{\updefault}$\Gamma_{1}$}}}}
\put(8851,239){\makebox(0,0)[lb]{\smash{{\SetFigFont{10}{12.0}{\rmdefault}{\mddefault}{\updefault}$k\in \C$}}}}
\put(2551,-136){\makebox(0,0)[lb]{\smash{{\SetFigFont{10}{12.0}{\rmdefault}{\mddefault}{\updefault}$\Gamma$}}}}
\put(2401,-3811){\makebox(0,0)[lb]{\smash{{\SetFigFont{10}{12.0}{\rmdefault}{\mddefault}{\updefault}$\Gamma_{4}$}}}}
\put(9076,-3811){\makebox(0,0)[lb]{\smash{{\SetFigFont{10}{12.0}{\rmdefault}{\mddefault}{\updefault}$\Gamma_{2}$}}}}
\put(11476,-1336){\makebox(0,0)[lb]{\smash{{\SetFigFont{10}{12.0}{\rmdefault}{\mddefault}{\updefault}$X$}}}}
\put(-149,-1336){\makebox(0,0)[lb]{\smash{{\SetFigFont{10}{12.0}{\rmdefault}{\mddefault}{\updefault}$- X$}}}}
\put(3676,-1336){\makebox(0,0)[lb]{\smash{{\SetFigFont{10}{12.0}{\rmdefault}{\mddefault}{\updefault}$- R \ell$}}}}
\put(7576,-1261){\makebox(0,0)[lb]{\smash{{\SetFigFont{10}{12.0}{\rmdefault}{\mddefault}{\updefault}$R \ell$}}}}
\end{picture}%
\end{center}
\caption{The paths $\Gamma_{j}$.}
\label{paf}
\end{figure}
\begin{equation}
\frac{1}{2 \pi i} \int_{-X+i\nu}^{X+i\nu} e^{-ikt} \widetilde{R}_{\chi}^{\ell} (k) u \, d k =\sum_{\fract{\lambda_j\in {\rm Res} \, P_{\ell}}{\im \lambda_j>-\mu}}\sum_{k=0}^{m(\lambda_j)}e^{-i\lambda_j t}t^k \pi_{j,k}^{\chi} u + \sum_{j=1}^5 \frac{1}{2 \pi i} \int_{\Gamma_j}e^{-it\lambda}\widetilde{R}_{\chi}^{\ell} (\lambda) \, d\lambda.  \label{3.5}
\end{equation}
Let $I_j= \frac{1}{2 \pi i} \int_{\Gamma_j}e^{-it\lambda}\widetilde{R}_{\chi}^{\ell} (\lambda) \, d\lambda$. We have, for $t$ large enough,
\begin{align}
\Vert I_3\Vert_{{\mathcal E}_{\ell}^{{\rm mod},-2}}&\lesssim \int_{X-i\ln \langle X \rangle}^{X+i\nu}
\Vert e^{-ist}\widetilde{R}_{\chi}^{\ell} (s) u \Vert_{{\mathcal E}_{\ell}^{{\rm mod} ,-2}} \, d s  \nonumber \\
\label{3.8}
&\lesssim \int_{-\ln \langle X \rangle}^{\nu} \frac{1}{\langle X \rangle^{2}} e^{(t - C)s} d s \, \Vert u\Vert_{{\mathcal E}^{{\rm mod},-2}_{\ell}} \lesssim \frac{e^{t \nu}}{t} X^{-2}\Vert u \Vert_{{\mathcal E}^{{\rm mod}}_{\ell}}.
\end{align}
We now take the limit $X$ goes to $+ \infty$ in the ${\mathcal L}({\mathcal E}^{\rm mod}_{\ell};{\mathcal E}_{\ell}^{{\rm mod},-2})$ sens in \eqref{3.5}. The integrals $I_{3}$ and $I_{5}$ go to $0$ thanks to \eqref{3.8} and, in the integrals $I_2$ and $I_4$, the paths $\Gamma_{\bullet}$ are replaced by paths which extend $\Gamma_{\bullet}$ in a natural way and which go to $\infty$. We note them again $\Gamma_{\bullet}$. We remark that
\begin{equation}
\int_{\Gamma_{4} \cup \Gamma_{1} \cup \Gamma_{2}} \frac{B_{j}}{(k-i(\nu+1) )^{-j-1}}e^{-ikt}dk=0,
\end{equation}
where the integral is absolutely convergent in ${\mathcal L}({\mathcal E}^{\rm mod}_{\ell};{\mathcal E}_{\ell}^{{\rm mod},-2})$. On the other hand, we have the estimate, for $t$ large enough,
\begin{align}
\Vert I_1 \Vert_{{\mathcal E}_{\ell}^{{\rm mod}}} &\lesssim \int_{-R \ell}^{R \ell}\Vert e^{-\mu t}\widehat{R}_{\chi}^{\ell} (s-i\mu) u \Vert_{{\mathcal E}_{\ell}^{{\rm mod}}} \, d s \nonumber   \\
&\lesssim e^{-\mu t}\int_{-R \ell}^{R \ell} \langle s \rangle^{C \mu} d s \, \Vert u\Vert_{{\mathcal E}_{\ell}^{\rm mod}}\lesssim e^{-\mu t} \ell^{C\mu} \Vert u \Vert_{{\mathcal E}_{\ell}^{\rm mod}},  \label{a62}  \\
\Vert I_2 \Vert_{{\mathcal E}_{\ell}^{{\rm mod}}} &\lesssim \int_0^{+ \infty} \Big\Vert e^{-i(R \ell + s-i(\mu+\ln \langle s \rangle))t}\widehat{R}_{\chi}^{\ell} (R \ell +s-i(\mu+\ln \langle s \rangle ))u \Big\Vert_{\mathcal{E}_{\ell}^{{\rm mod}}} d s  \nonumber \\
&\lesssim \int_0^{\infty}e^{-\mu t}e^{-\ln \langle s \rangle t}e^{C(\ln \langle s \rangle+\mu)} \, d s \, \Vert u\Vert_{{\mathcal E}_{\ell}^{\rm mod}}\lesssim e^{-\mu t}\Vert u \Vert_{{\mathcal E}_{\ell}^{\rm mod}}  ,  \label{a63}
\end{align}
and a similar estimate holds for $I_{4}$. Since all these estimates hold in ${\mathcal L} ({\mathcal E}_{\ell}^{{\rm mod}})$, \eqref{a62} and \eqref{a63} give the estimate of the rest \eqref{a68} with $M= C \mu /2$. The estimate \eqref{sp} follows from \eqref{pr}, Theorem \ref{a44} {\it iii)} and Proposition \ref{prop2.4}.

Let us now show part $(ii)$ of the theorem. We choose
$0>-\mu>\sup\{ \im \lambda ; \ \lambda \in ( {\rm Res} \, P ) \setminus \{0\}\}$ and the
integration path as in part $(i)$ of the theorem. We first suppose 
$e^{\varepsilon' t}>R \ell$ for some $\varepsilon'>0$ to be chosen later. Then
the estimate for $I_1$ can be replaced
by
\begin{eqnarray*}
\Vert I_1\Vert_{{\mathcal E}_{\ell}^{{\rm mod}}}\lesssim
e^{(C \mu \varepsilon'-\mu)t}\Vert u \Vert_{{\mathcal E}_{\ell}^{{\rm mod}}}.
\end{eqnarray*}
Let us now suppose $R \ell \ge e^{\varepsilon't}$. On the one hand we have
the inequality:
\begin{eqnarray*}
\Vert \chi e^{-itL}\chi\Vert_{{\mathcal L}({\mathcal E}^{\rm mod}_{\ell})}\lesssim 1 ,
\end{eqnarray*}
since the norms on ${\mathcal E}^{\rm mod}_{\ell}$ and on ${\mathcal E}_{\ell}$ are uniformly equivalent for $\ell \geq 1$. On the other hand by the hypotheses on $g$ we have 
\begin{eqnarray*}
1 \leq \frac{g(e^{2\varepsilon'' t})}{g( \ell ( \ell +1))}.
\end{eqnarray*}
for $\varepsilon'>\varepsilon''>0$ and $t$ large enough. It follows:
\begin{eqnarray*}
\Vert \chi e^{-itL}\chi\Vert_{{\mathcal L}({\mathcal
    E}^{\rm mod}_{\ell})} \lesssim \frac{g(e^{2\varepsilon'' t})}{g( \ell ( \ell +1))}.
\end{eqnarray*}
This finishes the proof of the theorem if we choose
$\varepsilon'$ sufficiently small and put $\varepsilon:=\min\{ 2 \varepsilon'',
  \mu- C \mu \varepsilon'\}$.

\begin{proof}[Proof of Remark \ref{a69} d)]
We note that for $u_{\ell} \in D (P_{\ell} )$, we have
\begin{align}
\< P_{\ell} u_{\ell} , u_{\ell} \> =& \big\< \big( r^{-1} D_{x} r^{2} D_{x} r^{-1} + V \ell ( \ell +1) \big) u_{\ell} , u_{\ell} \big\>  \nonumber  \\
\geq& \< V \ell ( \ell +1) \big) u_{\ell} , u_{\ell} \>,
\end{align}
and then
\begin{equation}  \label{a67}
\Vert \ell \sqrt{V} u_{\ell} \Vert^{2} \leq \Vert (P +1) u_{\ell} \Vert^{2}.
\end{equation}
Estimate \eqref{a68} can be written
\begin{equation*}
\Vert E_1(t)\Vert_{{\mathcal E}^{{\rm mod}}} \lesssim e^{-\mu t}\Vert\langle
-\Delta_{\omega}\rangle^M \chi_{0} u\Vert_{{\mathcal E}^{\rm mod}} ,
\end{equation*}
with $\chi_{0} \in C^{\infty}_{0} (\R)$ and $\chi_{0} \chi = \chi$. Let $\chi_{j} \in C^{\infty}_{0} (\R)$, $j =1 , \ldots , 2M$ with $\chi_{j+ 1} \chi_{j} = \chi_{j}$ for $j =0 , \ldots ,2M-1$. Remark that there exists $C>0$ such that $\sqrt{V} > 1/C$ on the support of $\chi_{2M}$. Using the radial decomposition $u = \sum_{\ell} u_{\ell}$, we get
\begin{align}
\Vert\< -\Delta_{\omega}\rangle^M \chi_{0} u\Vert_{{\mathcal E}^{\rm mod}} \lesssim& \sup_{\ell} \Vert \ell^{2M} \chi_{0} u_{\ell} \Vert_{{\mathcal E}^{\rm mod}}    \nonumber \\
\lesssim& \sup_{\ell} \Vert \ell^{2M-1} (P +1) \chi_{0} u_{\ell} \Vert_{{\mathcal E}^{\rm mod}} =\sup_{\ell} \Vert \ell^{2M-1} \chi_{1} (P +1) \chi_{0} u_{\ell} \Vert_{{\mathcal E}^{\rm mod}}  \nonumber \\
\lesssim& \sup_{\ell} \Vert \chi_{2M} (P+1) \chi_{2M-1} (P+1) \cdots \chi_{1} (P+1) \chi_{0} u_{\ell} \Vert_{{\mathcal E}^{\rm mod}}  \nonumber   \\
\lesssim& \Vert (P+1)^{2M} u \Vert_{{\mathcal E}^{\rm mod}} .
\end{align}
Finally, for the interpolation argument, we use the fact that
\begin{equation}
\Vert e^{- i t L_{\ell}} \Vert_{{\mathcal L}({\mathcal E}^{\rm mod}_{\ell};{\mathcal E}_{\ell}^{{\rm mod}})} \lesssim \Vert e^{- i t L_{\ell}} \Vert_{{\mathcal L}({\mathcal E}_{\ell};{\mathcal E}_{\ell})} = 1 ,
\end{equation}
for $\ell \geq 1$.
\end{proof}

\bibliographystyle{amsplain}
\providecommand{\bysame}{\leavevmode\hbox to3em{\hrulefill}\thinspace}
\providecommand{\MR}{\relax\ifhmode\unskip\space\fi MR }
\providecommand{\MRhref}[2]{%
  \href{http://www.ams.org/mathscinet-getitem?mr=#1}{#2}
}
\providecommand{\href}[2]{#2}

\end{document}